\documentclass[12pt,twoside,leqno]{article}
\usepackage{amsthm, amsfonts, amssymb, amsmath, amscd, enumerate}
\usepackage[colorlinks=true]{hyperref}
\usepackage[mathscr]{eucal}
\usepackage{mathrsfs}
\usepackage[all]{xy}
\usepackage{comment}
\usepackage{tikz-cd}
\usepackage{mathtools}
\theoremstyle{plain}
\newcommand{\CA}{{\mathcal {A}}}
\newcommand{\CB}{{\mathcal {B}}}

\newcommand{\CO}{{\mathcal {O}}}

\newcommand{\CX}{{\mathcal {X}}}
\newcommand{\CY}{{\mathcal {Y}}}
\newcommand{\CZ}{{\mathcal {Z}}}

\newcommand{\Char}{{\mathrm{char}}}
\newcommand{\Br}{{\mathrm{Br}}}

\newcommand{\Gal}{{\mathrm{Gal}}}

\newcommand{\Hom}{{\mathrm{Hom}}}

\newcommand{\Ker}{{\mathrm{Ker}}}

\newcommand{\NS}{{\mathrm{NS}}}

\newcommand{\rank}{{\mathrm{rank}}}
\newcommand{\Pic}{\mathrm{Pic}}

\renewcommand{\Im}{{\mathrm{Im}}}

\font\cyr=wncyr10  \newcommand{\Sha}{\hbox{\cyr X}}

\newcommand{\tor}{{\mathrm{tor}}}

\newcommand{\Tr}{{\mathrm{Tr}}}

\newcommand{\non}{{\mathrm{non}\text{-}}}

\newcommand{\Nm}{{\mathrm{Nm}}}

\DeclareMathOperator{\Spec}{Spec}

\newcommand{\QQ}{\mathbb{Q}}

\newcommand{\ZZ}{\mathbb{Z}} 
 
\newcommand{\GG}{\mathbb{G}}

\newcommand{\lra}{\longrightarrow}

\newcommand{\Coker}{\mathrm{Coker}} 
 
\newcommand{\et}{\mathrm{et}}

\newcommand{\SA}{\mathscr{A}} 
\newcommand{\SB}{\mathscr{B}}

\newtheorem{thm}{Theorem}[section]

\newtheorem{lem}[thm]{Lemma}
\newtheorem{prop}[thm]{Proposition}
\newtheorem{conj}[thm]{Conjecture}

\theoremstyle{definition}

\theoremstyle{remark}
\newtheorem{rem}[thm]{Remark}

\begin{document}
%------------------------------------------------------

\title{ On geometric Brauer groups and Tate-Shafarevich groups}

\author{Yanshuai Qin}

\maketitle
\tableofcontents
\newpage
\section{Introduction}
For any noetherian scheme $X$,  the \emph{cohomological Brauer group}
$$
\Br(X):=H^2(X,\mathbb{G}_m)_{\tor}
$$
is defined to be the torsion part of the etale cohomology group $H^2(X,\GG_m)$.
\subsection{Tate conjecture and Brauer group}
Let us first recall the Tate conjecture for divisors over finitely generated fields.
\begin{conj}\label{Tateconj} ( Conjecture $T^1(X,\ell)$)
Let $X$ be a projective and smooth variety over a finitely generated field $k$ of characteristic $p\geq0$, and $\ell\neq p$ be a prime number. Then the cycle class map\\
$$
\Pic(X)\otimes_\mathbb{Z}\mathbb{Q}_\ell\longrightarrow H^2(X_{k^s},\mathbb{Q}_\ell(1))^{G_k}
$$
is surjective.
\end{conj}
By a well-known result (see Proposition $\ref{prop2.1}$ ), $T^1(X,\ell)$ is equivalent to the finiteness of $\Br(X_{k^s})^{G_k}(\ell)$.
\subsection{Main theorems}
\begin{thm}\label{thm1}
Let $X$ be a smooth projective variety over a finitely generated field $K$ of characteristic $p>0$. Assuming that $\Br(X_{K^s})^{G_K}(\ell)$ is finite for some prime $\ell\neq p$, then $\Br(X_{K^s})^{G_K}(\non p)$ is finite.
\end{thm}
In the case that $K$ is finite, the above theorem was proved by Tate \cite[Thm. 5.2]{Tat3} and Lichtenbaum \cite{Lic}. The general case was proved by Cadoret-Hui-Tamagawa\cite{CHT}. Skorobogatov-Zarkhin \cite{SZ1,SZ2} proved the finiteness of $\Br(X_{k^s})^{G_k}(\non p)$ for abelian varieties and $K3$ surfaces($\Char(k)\neq 2$). In this paper, we will give an elementary proof of the theorem based on the idea of Tate and Lichtenbaum's proof for the case over finite fields.
\begin{thm}\label{thm2}
Let $\CY$ be a smooth irreducible variety over a finite field of characteristic $p>0$ with function field $K$. Let $A$ be an abelian variety over $K$. Define 
$$\Sha_{\CY}(A):=\Ker(H^1(K,A)\lra \prod_{s\in \CY^1}H^1(K_s^{sh}, A)),$$
where $\CY^1$ denotes the set of points of codimension  $1$ and $K_s^{sh}$ denotes the quotient field of a strict local ring at $s$. Assuming that $\Sha_{\CY}(A)(\ell)$ is finite for some prime $\ell\neq p$, then $\Sha_{\CY}(A)(\non p)$ is finite.
\end{thm}
In the case that $K$ is a global function field, the above theorem was proved by Tate \cite[Thm. 5.2]{Tat3} and Schneider\cite{Sch}. In \cite{Qin2}, we proved that the finiteness of $\Sha_{\CY}(A)(\ell)$ for a single $\ell$ is equivalent to the BSD conjecture for $A$ over $K$ formulated by Tate (cf. \cite{Tat1}).

\subsection{Notation and Terminology}
\subsubsection*{Fields}
By a \emph{finitely generated field}, we mean a field which is finitely generated over a prime field.\\
For any field $k$, denote by $k^s$ the separable closure. Denote by $G_k=\Gal(k^s/k)$ the absolute Galois group of $k$.
\subsubsection*{Henselization}
Let $R$ be a noetherian local ring, denote by $R^h$ (resp. $R^{sh}$) the henselization (resp. strict henselization) of $R$ at the maximal ideal. If $R$
is a domain, denote by $K^h$ (resp. $K^{sh}$) the fraction field of $R^{h}$ (resp. $R^{sh}$).
\subsubsection*{Varieties}
By a \emph{variety} over a field $k$, we mean a scheme which is separated and of finite type over $k$.

For a smooth proper geometrically connected variety $X$ over a field $k$, we use $\Pic^0_{X/k}$ to denote the identity component of the Picard scheme $\Pic_{X/k}$.
\subsubsection*{Cohomology}
The default sheaves and cohomology over schemes are with respect to the
small \'etale site. So $H^i$ is the abbreviation of $H_{\et}^i$.
\subsubsection*{Brauer groups}
For any noetherian scheme $X$, denote the \emph{cohomological Brauer group}
$$
\Br(X):=H^2(X,\mathbb{G}_m)_{\tor}.
$$
\subsubsection*{Abelian group}
For any abelian group $M$, integer $m$ and prime $\ell$, we set\\
$$M[m]=\{x\in M| mx=0\},\quad M_{\tor}=\bigcup\limits_{m\geq 1}M[m],\quad  M(\ell)=\bigcup\limits_{n\geq 1}M[\ell^n], $$
$$M\hat{\otimes}\ZZ_\ell=\lim \limits_{\substack{\leftarrow \\ n}}M/\ell^n, \quad T_\ell M=\Hom_\mathbb{Z}(\mathbb{Q}_\ell/\mathbb{Z}_\ell, M)=\lim \limits_{\substack{\leftarrow \\ n}}M[\ell^n],\quad V_\ell M= T_\ell(M)\otimes_{\mathbb{Z}_\ell}\mathbb{Q}_\ell.$$
For any torsion abelian group $M$ and prime number $p$, we set
$$M(\non p)=\bigcup\limits_{p\nmid m}M[m].$$ 
A torsion abelian group $M$ is \emph{of cofinite type} if $M[m]$ is finite for any positive integer $m$. We use $M_{div}$ to denote the maximal divisible subgroup of $M$. For a group $G$, we use $|G|$ to denote the order of $G$.

%%%%%%%%%%%%%%%%%%%%%%%%%%%%%%%%%%%%%%%%%%%%%%%%%%%%%%%%%%%%%%%%%%%%%%
%%%%%%%%%%%%% Acknowledgements
%%%%%%%%%%%%%%%%%%%%%%%%%%%%%%%%%%%%%%%%%%%%%%%%%%%%%%%%%%%%%%%%%%%%%%

\bigskip
\bigskip
\noindent {\bf Acknowledgements}:

\noindent I would like to thank Xinyi Yuan and Weizhe Zheng for important communications.
\section{Preliminary reductions}
\subsection{The Kummer sequence}
\begin{prop}\label{prop2.1}
Let $X$ be a smooth projective geometrically connected variety over a field $k$. Let $\ell$ be a prime different from char$(k)$, then the exact sequence of $G_k$-representations\\
$$
0\longrightarrow \NS(X_{k^s})/\ell^n\longrightarrow H^2(X_{k^s},\mathbb{Z}/\ell^n(1))\longrightarrow \Br(X_{k^s})[\ell^n]\longrightarrow 0
$$ 
is split  for sufficient large $\ell$. And for any $\ell\neq \Char(k)$, the exact sequence of $G_k$-representations
$$
0\longrightarrow \NS(X_{k^s})\otimes_\ZZ\mathbb{Q}_\ell\longrightarrow H^2(X_{k^s},\mathbb{Q}_\ell(1)) \longrightarrow V_\ell \Br(X_{k^s})\longrightarrow 0
$$
is split. Taking $G_k$-invariant, there is an exact sequence\\
$$
0\longrightarrow \NS(X)\otimes_\ZZ\mathbb{Q}_\ell\longrightarrow H^2(X_{k^s},\mathbb{Q}_\ell(1))^{G_k} \longrightarrow V_\ell \Br(X_{k^s})^{G_k}\longrightarrow 0.
$$
\end{prop}
\begin{proof}
Let $d$ denote the dimension of $X$. Let $Z_1(X_{k^s})$ denote the group of $1$-cycles on $X_{k^s}$, it admits a  natural $G_k$ action.
Since  $\tau$-equivalence is same as the numerical equivalence for divisors (cf., e.g. SGA 6 XIII, Theorem 4.6), thus the intersection pairing
$$
\NS(X_{k^s})_\mathbb{Q}\times Z_1(X_{k^s})_\mathbb{Q}\longrightarrow \mathbb{Q}
$$
is left non-degenerate. 
Since $\NS(X_{k^s})$ is finitely generated, so there exists a finite dimensional  $G_k$ invariant subspace $W$ of $Z^1(X_{k^s})_\mathbb{Q}$ such that the restriction of the intersection pairing to $\NS(X_{k^s})_\mathbb{Q}\times W$ is left non-degenerate. Since  $G_k$-actions factor through a finite quotient of $G_k$, we can choose $W$ such that the pairing is actually perfect. Let $W_{\mathbb{Q}_\ell}$ denote the subspace of $H^{2d-2}(X_{k^s},\mathbb{Q}_\ell(d-1))$ generated by cycle classes of $W$. Then the restriction of 
$$
H^2(X_{k^s},\mathbb{Q}_\ell(1))\times H^{2d-2}(X_{k^s},\mathbb{Q}_\ell(d-1)) \longrightarrow H^{2d}(X_{k^s},\mathbb{Q}_\ell(d)) \cong \mathbb{Q}_\ell
$$
to $\NS(X_{k^s})_{\mathbb{Q}_\ell}\times W_{\mathbb{Q}_\ell}$ is also perfect. So we have
$$
H^2(X_{k^s},\mathbb{Q}_\ell(1))=\NS(X_{k^s})_{\mathbb{Q}_\ell}\oplus W^{\perp}_{\mathbb{Q}_\ell}.
$$
This proves the second claim.

There exists a finitely generated $G_k$ invariant subgroup $W_0$ of $Z^1(X_{k^s})$ such that $W=W_0\otimes_\ZZ \mathbb{Q}$. Therefore, there exists a positive integer $N$, such that the base change of the pairing $\NS(X_{k^s})\times W_0 \longrightarrow \mathbb{Z}$ to $\mathbb{Z}[N^{-1}]$ is perfect. So for any $\ell \nmid N$, the intersection pairing
$$ 
 \NS(X_{k^s})/\ell^n \times W_0/\ell^n \longrightarrow \mathbb{Z}/\ell^n
$$
is perfect. Since it is compatible with 
$$
H^2(X_{k^s},\mathbb{Z}/\ell^n(1))\times H^{2d-2}(X_{k^s},\mathbb{Z}/\ell^n(d-1)) \longrightarrow H^{2d}(X_{k^s},\mathbb{Z}/\ell^n(d)) \cong \mathbb{Z}/\ell^n.
$$
Thus we have
$$
H^2(X_{k^s},\mathbb{Z}/\ell^n(1))=\NS(X_{k^s})/\ell^n\oplus (W_0/\ell^n)^{\perp}.
$$
This proves the first claim.
\end{proof}
\subsection{Reduce to the vanishing of Galois cohomology}
Let $X$ be a smooth projective geometrically connected variety over a finitely generated field $K$ of characteristic $p>0$. Assuming that $\Br(X_{K^s})^{G_K}(\ell)$ is finite for some prime $\ell\neq p$, by \cite[Cor. 1.7]{Qin2}, $\Br(X_{K^s})^{G_K}(\ell)$ is finite for all primes $\ell\neq p$. Thus, to prove Theorem \ref{thm1}, it suffices to prove $\Br(X_{K^s})^{G_K}[\ell]=0$ for $\ell \gg 0$ under the assumption of the finiteness of $\Br(X_{K^s})^{G_K}(\ell)$. By Proposition \ref{prop2.1}, for $\ell\gg 0$, there is an exact sequence
$$
0\longrightarrow (\NS(X_{K^s})/\ell)^{G_K}\longrightarrow H^2(X_{K^s},\mathbb{Z}/\ell(1))^{G_K}\longrightarrow \Br(X_{K^s})^{G_K}[\ell]\longrightarrow 0.
$$ 
Therefore, it suffices to show that $(\NS(X_{K^s})/\ell)^{G_K}\rightarrow H^2(X_{K^s},\mathbb{Z}/\ell(1))^{G_K}$ is surjective for $\ell\gg 0$. Consider the following exact sequence of $\ell$-adic sheaves on $X$
$$0\lra \ZZ_\ell(1)\stackrel{\ell}{\lra}\ZZ_\ell(1)\lra \ZZ/\ell(1)\lra 0.$$
It gives a long exact sequence
$$ H^2(X_{K^s}, \ZZ_\ell(1))\stackrel{\ell}{\lra} H^2(X_{K^s}, \ZZ_\ell(1))\lra H^2(X_{K^s}, \ZZ/\ell(1))\lra H^3(X_{K^s}, \ZZ_\ell(1))[\ell] \lra 0.
$$
By a theorem of Gabber (cf.\cite{Gab}), $H^3(X_{K^s}, \ZZ_\ell(1))[\ell]=0$ for $\ell\gg 0$. Thus, for $\ell \gg 0$, there is a canonical isomorphsim
$$H^2(X_{K^s}, \ZZ_\ell(1))/\ell\cong H^2(X_{K^s}, \ZZ/\ell(1)).$$
Assuming that $\Br(X_{K^s})^{G_K}(\ell)$ is finite, by the exact sequence
$$0\lra (\NS(X_{K^s})\otimes_{\ZZ}\ZZ_\ell)^{G_K}\lra H^2(X_{K^s}, \ZZ_\ell(1))^{G_K}\lra T_\ell\Br(X_{K^s})^{G_K},
$$
we have 
$$(\NS(X_{K^s})\otimes_{\ZZ}\ZZ_\ell)^{G_K}\cong H^2(X_{K^s}, \ZZ_\ell(1))^{G_K}.
$$
Thus, it suffices to show that the natural map
$$H^2(X_{K^s}, \ZZ_\ell(1))^{G_K}/\ell\lra (H^2(X_{K^s}, \ZZ_\ell(1))/\ell)^{G_K}
$$
is surjective for $\ell\gg 0$. By Gabber's theorem(cf.\cite{Gab}), $H^2(X_{K^s}, \ZZ_\ell(1))[\ell]=0$ for $\ell\gg 0$. Thus, for $\ell \gg 0$, there is an exact sequence
$$0\lra H^2(X_{K^s}, \ZZ_\ell(1))^{G_K}\stackrel{\ell}{\lra} H^2(X_{K^s}, \ZZ_\ell(1))^{G_K}\lra (H^2(X_{K^s}, \ZZ_\ell(1))/\ell)^{G_K}$$
$$\lra H^1_{cts}(G_K,H^2(X_{K^s}, \ZZ_\ell(1)))[\ell]\lra 0.$$
Therefore, the question is reduced to show
$$H^1_{cts}(G_K,H^2(X_{K^s}, \ZZ_\ell(1)))[\ell]=0 \ \mathrm{for} \ \ell \gg 0$$
under the assumption of the finiteness of $\Br(X_{K^s})^{G_K}(\ell)$.
\section{Proof of Theorem \ref{thm1}}
Let $X$ be a smooth projective geometrically connected variety over a finitely generated field $K$ of characteristic $p>0$. Assuming that $\Br(X_{K^s})^{G_K}(\ell)$ is finite for all prime $\ell\neq p$, we will show
\begin{equation}\label{key}
H^1_{cts}(G_K,H^2(X_{K^s}, \ZZ_\ell(1)))[\ell]=0 \ \mathrm{for} \ \ell \gg 0.
\end{equation}
\subsection{The spreading out trick}

\begin{lem}\label{spread}
Let $\CY$ be an affine smooth geometrically connected variety over a finite field $k$ with function field $K$. Let $\pi:\CX \lra \CY $ be a smooth projective morphism with a generic fiber $X$ geometrically connected over $K$. Let $\ell\neq \Char(k)$ be a prime. Then the following statements are true.
\begin{itemize}
\item[(a)]	
The spectral sequence 
$$E^{p,q}_2=H^p(\CY_{k^s},R^q\pi_*\ZZ_\ell(1))\Rightarrow H^{p+q}(\CX_{k^s},\ZZ_\ell(1))
$$
degenerates at $E_2$ for $\ell\gg 0$. 
\item[(b)]
There is an exact sequence 
$$0\lra H^1(\CY, R^2\pi_*\ZZ_\ell(1))\lra H^1_{cts}(K,H^2(X_{K^s}, \ZZ_\ell(1)))\lra \prod_{y\in \CY}H^1_{cts}(K_{y}^{sh},H^2(X_{K^s}, \ZZ_\ell(1))).
$$
\item[(c)]
For $\ell \gg 0$, we have 
$$H^1_{cts}(K,H^2(X_{K^s}, \ZZ_\ell(1)))[\ell] \cong H^1(\CY, R^2\pi_*\ZZ_\ell(1))[\ell].$$
\item[(d)] There is an exact sequence
$$0\lra H^0(\CY_{k^s}, R^2\pi_*\ZZ_\ell(1))_{G_k}\lra H^1(\CY, R^2\pi_*\ZZ_\ell(1))\lra H^1(\CY_{k^s}, R^2\pi_*\ZZ_\ell(1))^{G_k}.$$

\end{itemize}
\end{lem}
\begin{proof}
To prove $(a)$, we will use Deligne's Lefschetz criteria for degeneration of spectral sequences (cf.\cite[Prop. 2.4]{Del1}). Let $u\in H^2(\CX,\ZZ_\ell(1))$ be the $\ell$-adic chern class of the ample line sheaf $\CO(1)$ associated to the projective morphism $\pi$. It suffices to check $u$ and $\ZZ_\ell(1)$ satisfies the Lefschetz condition. Since $R^q\pi_*\ZZ_\ell(1)$ is lisse for all $q$, by the proper base change theorem, it suffices to check the Lefschetz condition for a single fiber of $\pi$ over a closed point $y\in \CY$. Thus, we might assume that $\CY=\Spec k$. Let $m$ denote $\dim (\CX)$.  By Gabber's Theorem \cite[\S 6]{Gab}, the Lefschetz morphism
$$H^{m-i}(\CX_{k^s},\ZZ_\ell(1))\lra H^{m+i}(\CX_{k^s},\ZZ_\ell(i+1))
$$ 
induced by $u^i$ is an isomorphism for $0\leq i\leq m$ and $\ell \gg 0$. This proves the first claim.

For $(b)$, let $j:\Spec K\rightarrow \CY$ be the generic point. Since $\pi$ is smooth and proper, $R^2\pi_*\ZZ/\ell^n(1)$ is a locally constant sheaf for all $n\geq 1$. Thus,we have
$$R^2\pi_*\ZZ/\ell^n(1)\cong j_*j^*R^2\pi_*\ZZ/\ell^n(1).$$
By the spectral sequence
$$
E_2^{p,q}=H^p(\CY,R^qj_*j^*R^2\pi_*\ZZ/\ell^n(1))\Rightarrow H^{p+q}(\Spec K, j^*R^2\pi_*\ZZ/\ell^n(1)),
$$
we get an exact sequence
$$                                    
0\lra H^1(\CY,j_*j^*R^2\pi_*\ZZ/\ell^n(1))\lra H^1(K, j^*R^2\pi_*\ZZ/\ell^n(1)) \lra H^0(\CY,R^1j_*j^*R^2\pi_*\ZZ/\ell^n(1)).
$$
Thus, there is an exact sequence
$$0\lra H^1(\CY, R^2\pi_*\ZZ/\ell^n(1))\lra H^1(K,H^2(X_{K^s}, \ZZ/\ell^n(1)))\lra \prod_{y\in \CY}H^1(K_{y}^{sh},H^2(X_{K^s}, \ZZ/\ell^n(1))).
$$
Taking limit, we get $(b)$.

For $(c)$, it suffices to show that the third term of the sequence in $(b)$ has no torsion for $\ell \gg 0$. Since $R^2\pi_*\ZZ_\ell(1)$ is lisse, the action of $\Gal(K^s/K^{sh}_y)$ on $H^2(X_{K^s},\ZZ_\ell(1))$ is trivial. Thus,
$$H^1_{cts}(K_y^{sh},H^2(X_{K^s},\ZZ_\ell(1)))=\Hom_{cts}(\Gal(K^s/K^{sh}_y),H^2(X_{K^s},\ZZ_\ell(1))).$$
So the claim follows from the fact $H^2(X_{K^s},\ZZ_\ell(1))[\ell]=0$ for $\ell \gg 0$.

$(d)$ follows from the Hochschild-Serre spectral sequence
$$E_2^{p,q}=H^p_{cts}(G_k,H^q(\CY_{k^s}, R^2\pi_*\ZZ_\ell(1)))\Rightarrow H^{p+q}(\CY, R^2\pi_*\ZZ_\ell(1))$$
and the fact $H^1_{cts}(G_k, M)=M_{G_k}$.
\end{proof}
\subsection{Compatible system of $G_k$-modules}
By Lemma \ref{spread} $(c)$ and $(d)$, to prove (\ref{key}), it suffices to show 
$$ H^0(\CY_{k^s}, R^2\pi_*\ZZ_\ell(1))_{G_k}[\ell]=0 \ \mathrm{and} \ H^1(\CY_{k^s}, R^2\pi_*\ZZ_\ell(1))^{G_k}[\ell]=0 \ \mathrm{for} \ \ell \gg 0 .$$

Let $k$ be a finite field of characteristic $p>0$ and $I$ be the set of primes different from $p$. Let $F\in G_k$ be the geometric Frobenius element. Let $M=(M_\ell, \ell\in I)$ be a family of finitely generated $\ZZ_\ell$ -modules equipped with continuous $G_k$-actions. We say that $M=(M_\ell, \ell\in I)$ is a \emph{compatible system of $G_k$-modules} if
there exists a polynomial $P(T)\in \QQ[T]$ with $P(1)\neq 0$ such that $(F-1)P(F)$ kills all $M_\ell\otimes_{\ZZ_\ell}\QQ_\ell$. Let $N=(N_\ell,\ell \in I)$ be a family of $G_k$-modules. We say that $N$ is a system of submodules (resp. quotient modules) of $M$ if $N_\ell$ is a $G_k$-submodule (resp. $G_k$-quotient modules) of $M_\ell$ for all $\ell \in I$. It is obvious that a system of submodules (resp. quotient modules) of a compatible system of $G_k$-modules is also a compatible system.

We will show that $H^0(\CY_{k^s}, R^2\pi_*\ZZ_\ell(1)), \ell \in I$ and 
$H^1(\CY_{k^s}, R^2\pi_*\ZZ_\ell(1)), \ell \in I$ are compatible systems of $G_k$-modules.
\begin{lem}\label{comp}
Let $M$ be a compatible system of $G_k$-modules and $f_\ell:M_\ell^{G_k} \rightarrow (M_\ell)_{G_k}$ be the map induced by the identity. Then $f$ has a finite kernel and a finite cokernel and
$$ |\Ker(f_\ell)|=|\Coker(f_\ell)| \ \mathrm{for} \ \ell \gg 0.
$$
Moreover, if $1$ is not an eigenvalue of $F$ on $M_\ell\otimes_{\ZZ_\ell}\QQ_\ell$ for all $\ell \in I$, then
$$|M_\ell^{G_k}|=|(M_\ell)_{G_k}|<\infty \ \mathrm{for} \ \ell \gg 0.$$

\end{lem}
\begin{proof}
Let $\theta_\ell$ denote the endomorphsim $F-1$ on $M_\ell$ and $\theta_\ell\otimes 1$ denote the corresponding endomorphsim of $M_\ell\otimes_{\ZZ_\ell}\QQ_\ell$. Write $\det(T-\theta_\ell\otimes 1)=T^{\rho_\ell} R_\ell(T)$ with $R_\ell(0)\neq 0$. By the assumption, there exists a polynomial $P(T)\in \QQ[T]$ with $P(1)\neq 0$ such that $(F-1)P(F)$ kills all $M_\ell\otimes_{\ZZ_\ell}\QQ_\ell$. Thus, the generalized $1$-eigenspace of $F$ on $M_\ell\otimes_{\ZZ_\ell}\QQ_\ell$ is equal to $\Ker(F-1)$. Thus, $\rho_\ell=\rank_{\ZZ_\ell}(\Ker(\theta_\ell))$. By \cite[\S 5, Lem. z.4]{Tat3}, $f_\ell$ has a finite kernel and a finite cokernel and
$$|\Ker(f_\ell)|/|\Coker(f_\ell)|=|R_\ell(0)|_\ell.$$
$R_\ell(0)$ is equal to the product of roots of $R_\ell(T)$ and all roots of $R_\ell(T)$ are roots of $P(T+1)$. Since all roots of $P(T+1)$ have $\ell$-adic absolute values $1$ for $\ell\gg 0$, so $|R_\ell(0)|_\ell=1$ for $\ell \gg 0$. This completes the proof of the first statement.

For the second statement, under the assumption, $f_\ell$ is a map between two finite groups. Then the claim follows from the first statement and the following exact sequence of finite abelian groups
$$0\lra \Ker(f_\ell) \lra M_\ell^{G_k} \lra (M_\ell)_{G_k}\lra \Coker(f_\ell) \lra 0.
$$

\end{proof}
\begin{lem}\label{lem1}
Assume that $\Br(X_{K^s})^{G_K}(\ell)$ is finite for all $\ell \in I$. Then 
$H^0(\CY_{k^s}, R^2\pi_*\ZZ_\ell(1)), \ell \in I$ is a compatible system of $G_k$-modules. Moreover,
$$H^0(\CY_{k^s}, R^2\pi_*\ZZ_\ell(1))_{G_k}[\ell]=0 \ \mathrm{for} \ \ell \gg 0 .
$$
\end{lem}
\begin{proof}
Let $K^\prime$ denote $Kk^s$. Since $R^2\pi_*\QQ_\ell(1)$ is lisse, we have 
$$H^0(\CY_{k^s}, R^2\pi_*\QQ_\ell(1))=H^2(X_{K^s},\QQ_\ell(1))^{G_{K^\prime}}.$$
By Proposition \ref{prop2.1}, there is a split exact sequences of $G_k$-representations
$$0\lra (\NS(X_{K^s})\otimes_{\ZZ}\QQ_\ell)^{G_{K^\prime}}\lra H^2(X_{K^s}, \QQ_\ell(1))^{G_{K^\prime}}\lra V_\ell\Br(X_{K^s})^{G_{K^\prime}}\lra 0.
$$
Write $W_\ell$ for $H^2(X_{K^s}, \QQ_\ell(1))^{G_{K^\prime}}$. By the assumption, $V_\ell\Br(X_{K^s})^{G_K}=0$. Thus, $W_\ell^{G_k}$ is equal to $(\NS(X_{K^s})\otimes_{\ZZ}\QQ_\ell)^{G_{K}}$ which is a direct summand of $W_\ell$ as $G_k$-representations. This implies that the generalized $1$-eigenspace of $F$ on $W_\ell$ is equal to $\Ker(F-1)$. By Lemma \ref{spread} $(a)$, $W_\ell$ is a subquotient of $H^2(\CX_{k^s},\QQ_\ell(1))$. By Proposition \ref{lindep} below, there exists a non-zero polynomial $P(T)\in \QQ[T]$ such that $P(F)$ kills all $W_\ell$. Write $P(T)=(T-1)^mP_1(T)$ with $P(1)\neq 0$. Then $(F-1)P_1(F)$ kills all $W_\ell$. This proves the first claim.

For the second claim, set $M_\ell=H^0(\CY_{k^s}, R^2\pi_*\ZZ_\ell(1))$. $M_\ell\cong H^2(X_{K^s},\ZZ_\ell(1))^{G_{K^\prime}}$, so $(M_\ell)_{tor}=0$ for $\ell \gg 0$. Let $f_\ell:M_\ell^{G_k} \rightarrow (M_\ell)_{G_k}$ be the map in Lemma \ref{comp}. By Lemma \ref{comp}, $\Ker(f_\ell)$ is a finite subgroup of $M_\ell$, thus $\Ker(f_\ell)=0$ for $\ell \gg 0$. By Lemma \ref{comp}, this implies $\Coker(f_\ell)=0$ for $\ell \gg 0$. So $f_\ell$ is an isomorphism for $\ell \gg 0$. Thus, $(M_\ell)_{G_k}[\ell]\cong M_\ell^{G_k}[\ell]=0$ for $\ell \gg 0$.

\end{proof}

\begin{lem}\label{lem2}
Let $H^1(\CY_{k^s}, R^2\pi_*\ZZ_\ell(1))$ as in Lemma \ref{spread}. Then
$$H^1(\CY_{k^s}, R^2\pi_*\ZZ_\ell(1))^{G_k}[\ell]=0 \ \mathrm{for} \ \ell \gg 0.
$$
\end{lem}
\begin{proof}
Write $H^3_\ell$ for $H^3(\CX_{k^s},\ZZ_\ell(1))$ in Lemma \ref{spread} $(a)$. And let $0 \subseteq (H^3_\ell)^3\subseteq (H^3_\ell)^2\subseteq(H^3)^1\subseteq (H^3_\ell)^0=H^3_\ell$ be the filtration induced by the spectral sequence in Lemma \ref{spread} $(a)$. Since the spectral sequence degenerates at $E_2$, we have 
$$H^1(\CY_{k^s}, R^2\pi_*\ZZ_\ell(1))\cong (H^3_\ell)^1/(H_\ell^3)^2.$$
By Proposition \ref{lindep} below, $H^3_\ell, \ell\in I$ is a compatible system of $G_k$-modules. Thus, $(H_\ell^3)^1, \ell \in I$ and $(H_\ell^3)^2, \ell \in I$ \ are also compatible systems of $G_k$-modules. Since $H^3_\ell\otimes_{\ZZ_\ell}\QQ_\ell$ is of weight $\geq 1$, $1$ is not an eigenvalue for $F$. By Lemma \ref{comp}, we have 
$$|((H_\ell^3)^2)_{G_k}|=((H_\ell^3)^2)^{G_k}| \ \mathrm{for} \ \ell \gg 0.
$$
By Theorem \ref{bigthm} below, 
$$
(H^3_\ell)^{G_k} =0\ \mathrm{for} \ \ell \gg 0.
$$
It follows that
$$
((H_\ell^3)^2)^{G_k}=((H_\ell^3)^1)^{G_k}=(H^3_\ell)^{G_k} =0\ \mathrm{for} \ \ell \gg 0.
$$
Thus,
$$
((H_\ell^3)^2)_{G_k}=0 \ \mathrm{for} \ \ell \gg 0.
$$
By the exact sequence
$$0\lra ((H_\ell^3)^2)^{G_k}\lra ((H_\ell^3)^1)^{G_k}\lra ((H^3_\ell)^1/(H_\ell^3)^2)^{G_k} \lra ((H_\ell^3)^2)_{G_k},$$
we can conclude
$$
((H^3_\ell)^1/(H_\ell^3)^2)^{G_k}=0 \ \mathrm{for} \ \ell \gg 0.
$$
\end{proof}
\subsection{Proof of Theorem \ref{thm1}}
Lemma \ref{lem1} and Lemma \ref{lem2} completes the proof of Theorem \ref{thm1}.

\section{Torsions of cohomology of varieties over finite fields}
In this section, we will study the torsion part of $l$-adic cohomology for smooth varieties over finite fields.
\subsection {Alteration}
By de Jong's alteration theorem (cf. \cite{deJ}), every integral variety $X$ over an algebraic closed field $k$ admits a proper and generic finite morphism $f:X_1\rightarrow X$ such that $X_1$ is an open subvariety of a smooth projective connected variety over $k$. The alteration theorem will allow us to reduce questions about general non-compact smooth varieties to smooth varieties with smooth projective compactification. 
\begin{lem}\label{alter}
Let $f:Y\rightarrow X$ be a proper and generic finite morphism between smooth irreducible varieties over an algebraic closed field $k$.  Let $\ell \neq \Char(k)$ be a prime and $d$ denote $[K(Y):K(X)]$. Then the kernels of the pullbacks 
$$H^i(X,\ZZ/\ell^n)\lra H^i(Y,\ZZ/\ell^n)\ \mathrm{and}\ H^i_c(X,\ZZ/\ell^n)\lra H^i_c(Y,\ZZ/\ell^n)$$
are killed by $d$ for all integers $i\geq 0$ and $n\geq 1$. 
\end{lem}
\begin{proof}
Let $m$ denote the dimension of $X$. Since $X$ is smooth and irreducible, there is a perfect Poincar\'e duality pairing
$$H^i(X,\ZZ/\ell^n)\times H_c^{2m-i}(X,\ZZ/\ell^n(m)) \lra H^{2m}_c(X,\ZZ/\ell^n(m)).
$$
Since the pullback $f^*$ map is compatible with the above pairing, it suffices to show that the kernel of the pullback map
$$H^{2m}_c(X,\ZZ/\ell^n(m))\lra H^{2m}_c(Y,\ZZ/\ell^n(m))
$$
is killed by $d$. Thus, without loss of generality, we can shrink $X$ such that $f$ is finite flat. Then the above map can be identified with the multiplication by $d$
$$\ZZ/\ell^n \stackrel{d}{\lra}\ZZ/\ell^n$$
through trace maps for $X$ and $Y$. This proves the claim.

\end{proof}
\begin{prop}\label{notor}
Let $X$ be a smooth variety over an algebraic closed field $k$. Then, for $0\leq i\leq 2$,{}
$$H^i(X,\ZZ_\ell)_{tor}=0 \ \mathrm{for} \ \ell \gg 0.
$$
\end{prop}
\begin{proof}
In the case that $X$ is smooth and projective over $k$, this is Gabber's theorem. We might assume that $X$ is irreducible. Then $X$ admits an alteration $f:Y\lra X$ such that $Y$ is an open subvariety of a smooth projective connected variety $\bar{Y}$. By Lemma \ref{alter}, the kernel of
$$H^i(X,\ZZ_\ell)\lra H^i(Y,\ZZ_\ell)$$
is killed by the degree of $f$ and therefore, is injective for $\ell \gg 0$.
So it suffices to prove the claim for $Y$. For $i=0$, the claim is obvious. Let $Z$ denote $\bar{Y}-Y$ and $Z_{i}$ denote the smooth locus of irreducible components of $Z$ with codimension $1$ in $\bar{Y}$. By the theorem of purity (cf. \cite{Fuj}), there is an exact sequence
$$0\lra H^1(\bar{Y},\ZZ_\ell)\lra H^1(Y,\ZZ_\ell)\lra \bigoplus_i H^0(Z_i, \ZZ_\ell(-1)).
$$
Since the first and the third term are torsion-free for $\ell \gg 0$, so $H^1(Y,\ZZ_\ell)_{tor}=0$ for $\ell \gg 0$. This proves the claim for $i=1$. For $i=2$, consider the exact sequence
$$0\lra \Pic(Y)\hat{\otimes}\ZZ_\ell\lra H^2(Y,\ZZ_\ell(1))\lra T_\ell\Br(Y)\lra 0.$$
Since the third term is always torsion free, it suffices to show that
$\Pic(Y)\hat{\otimes}\ZZ_\ell$ is torsion free for $\ell\gg 0$. Let $a:\Pic(\bar{Y})\rightarrow \Pic(Y)$ be the restriction map. Since $\Pic^0(\Bar{Y})$ is $\ell$-divisible for all $\ell \neq \Char(k)$, so $a(\Pic^0(\Bar{Y}))$ is also $\ell$-divisible. Thus
$$\Pic(Y)\hat{\otimes}\ZZ_\ell=\Pic(Y)/a(\Pic^0(\Bar{Y}))\hat{\otimes}\ZZ_\ell.$$
Since $\NS(Y)\rightarrow \Pic(Y)/a(\Pic^0(\Bar{Y}))$ is surjective, $\Pic(Y)/a(\Pic^0(\Bar{Y}))$ is a finitely generated abelian group. Thus, the torsion part of $\Pic(Y)/a(\Pic^0(\Bar{Y}))\hat{\otimes}\ZZ_\ell$ is isomorphic to the $\ell$-primary torsion part of $\Pic(Y)/a(\Pic^0(\Bar{Y}))$, which is zero for $\ell \gg 0$. This completes the proof.

\end{proof}

\subsection{$\ell$-independence }

\begin{prop}\label{lindep}
Let $X$ be a separated scheme of finite type over a finite field $k$. Let $F\in G_k$ be the geometric Frobenius element. Then there exists a non-zero polynomial $P(T)\in \QQ[T]$ such that $P(F)$ kills $H^i_c(X_{k^s},\QQ_\ell)$ for all primes $\ell \neq \Char(k)$ and $i\geq 0$. And the same claim holds for $H^i(X_{k^s},\QQ_\ell)$ when $X$ is smooth over $k$. 
\end{prop}
\begin{proof}
Note that the second claim follows from the first one and Poincar\'e duality. For the first claim, in the case that $X$ is smooth proper over $k$, this is a consequence of Deligne's theorem in Weil I. For the general case, we will prove it by induction on the dimension of $X$. Firstly, we note that replacing $k$ by a finite extension will not change the question and we can assume that $X$ is integral. Let $U$ be the regular locus of $X$ and $Y=X-U$. There is a long exact sequence
$$H^{i}_c(U_{k^s},\QQ_\ell)\lra H^i_c(X_{k^s},\QQ_\ell)\lra H^i_c(Y_{k^s},\QQ_\ell) \lra H^{i+1}_c(U_{k^s},\QQ_\ell)
$$
By induction, the claim is true for $Y$. Thus, we can assume that $X$ is smooth. Let $X_1\lra X$ be an alteration such that $X_1$ admits a smooth projective compactification. By Lemma \ref{alter}, the pullback map
$$H^i_c(X_{k^s},\QQ_\ell) \lra H^i_c((X_1)_{k^s},\QQ_\ell)$$
is injective. It suffices to prove the claim for $X_1$. Thus, we can assume that $X$ is an open subvariety of smooth projective irreducible variety $\bar{X}$ over $k$. Let $Y$ denote $\bar{X}-X$. Consider the 
long exact sequence
$$H^{i-1}_c(Y_{k^s},\QQ_\ell)\lra H^i_c(X_{k^s},\QQ_\ell)\lra H^i_c(\bar{X}_{k^s},\QQ_\ell) \lra H^{i}_c(Y_{k^s},\QQ_\ell)
.$$
By induction, the claim is true for $Y$. Since the claim is true for $\bar{X}$, it follows that the claim holds for $X$. This completes the proof.
\end{proof}
\begin{thm}\label{bigthm}
Let $X$ be a smooth variety over a finite field $k$ of characteristic $p$. Then
$$H^3(X_{k^s},\ZZ_\ell(1))^{G_k}=0 \ \mathrm{for} \ \ell \gg 0.$$
\end{thm}
\begin{proof}
Note that replacing $k$ by a finite extension will not change the question. Thus, we can assume that $X$ is geometrically connected. There is an alteration (extending $k$ if necessary) $X_1\rightarrow X$ such that $X_1$ is an open subvariety of a smooth projective geometrically connected variety over $k$. By Lemma \ref{alter}, the pullback map 
$$H^3(X_{k^s},\ZZ_\ell(1)) \lra H^3((X_1)_{k^s},\ZZ_\ell(1))$$
is injective for $\ell \gg 0$. So it suffices to prove the claim for $X_1$. Thus, we can assume that $X$ is open subvariety of a smooth projective geometrically connected variety $\bar{X}$ over $k$. There is an open subset $U$ of $\bar{X}$ such that $U$ contains $X$ with $Y=U-X$ a smooth divisor in $U$ and $Z=\bar{X}-U$ is of codimension $\geq 2$ in $\bar{X}$. There is an exact sequence
$$H^3_{Z}(\bar{X}_{k^s},\ZZ_\ell(1))\lra H^3(\bar{X}_{k^s},\ZZ_\ell(1))\lra H^3(U_{k^s},\ZZ_\ell(1))\lra H^4_{Z}(\bar{X}_{k^s},\ZZ_\ell(1)).
$$
By the theorem of purity (cf. \cite{Fuj}), $H^3_{Z}(\bar{X}_{k^s},\ZZ_\ell(1))=0$. By the theorem of semi-purity (cf. \cite{Fuj}), removing a close subset of codimension $\geq 3$ in $\bar{X}$ from $Z$ will not change these groups. So we can assume that $Z$ is smooth and of pure codimension $2$ in $\CY$. Then, by the theorem of purity,
$$H^4_{Z}(\bar{X}_{k^s},\ZZ_\ell(1))\cong H^0(Z_{k^s},\ZZ_\ell(-1))\cong \oplus \ZZ_\ell(-1).$$
It folllows 
$$H^3(\bar{X}_{k^s},\ZZ_\ell(1))^{G_k}\cong H^3(U_{k^s},\ZZ_\ell(1))^{G_k}.$$
Since $H^3(\bar{X}_{k^s},\QQ_\ell(1))$ is of pure weight $1$, so $H^3(\bar{X}_{k^s},\ZZ_\ell(1))^{G_k}$ is a torsion group and therefore, vanishes for $\ell\gg 0$ by Gabber's theorem. Thus, the claim holds for $U$. There is an exact sequence
$$H^3_{Y}(U_{k^s},\ZZ_\ell(1))\lra H^3(U_{k^s},\ZZ_\ell(1))\stackrel{a_\ell}{\lra}H^3(X_{k^s},\ZZ_\ell(1))\lra H^4_{Y}(U_{k^s},\ZZ_\ell(1)).
$$
By the theorem of purity, $H^4_{Y}(U_{k^s},\ZZ_\ell(1))\cong H^2(Y_{k^s},\ZZ_\ell)$. Thus, $H^4_{Y}(U_{k^s},\ZZ_\ell(1))^{G_k}\cong H^2(Y_{k^s},\ZZ_\ell)^{G_k}$ is a torsion abelian group and therefore, vanishes for $\ell \gg 0$ by Proposition \ref{notor}. Thus, it suffices to prove $\Im(a_\ell)^{G_k}=0$ for $\ell \gg 0$. Since $H^3(U_{k^s},\QQ_\ell(1))$ is of weight $\geq 1$, by Proposition \ref{lindep}, $H^3(U_{k^s},\ZZ_\ell(1)), \ell \in I$ and $\Im(a_\ell), \ell \in I$ are compatible systems of $G_k$-modules. Let $M$ be one of these systems. By Lemma \ref{comp}, the vanishing of $M_{\ell}^{G_k}$ and $(M_\ell)_{G_k}$ are equivalent for $\ell \gg 0$. Since $H^3(U_{k^s},\ZZ_\ell(1))^{G_k}=0$ for $\ell \gg 0$, we have $H^3(U_{k^s},\ZZ_\ell(1))_{G_k}=0$ for $\ell \gg 0$. Thus, $\Im(a_\ell)_{G_k}=0$ for $\ell \gg 0$. It follows that
$\Im(a_\ell)^{G_k}=0$ for $\ell \gg 0$. This completes the proof.
\end{proof}
\section{Proof of Theorem \ref{thm2}}
\subsection{Theorem \ref{thm1} for non-compact varieties}
To prove Theorem \ref{thm2}, we need a version of Theorem \ref{thm1} for non-compact smooth varieties over finite fields.

\begin{lem}\label{surj} Let $X$ be a smooth irreducible variety over a finite field $k$ of characteristic $p$. Assuming that $X$ admits a finite flat alteration $X_1\lra X$ such that $X_1$ is an open subvariety of a smooth projective variety over $k$, then the natural map
$$\Br(X)(\non p)\lra\Br(X_{k^s})^{G_k}(\non p)$$
has a cokernel of finite exponent.
\end{lem}
\begin{proof}
Let $f:X_1\rightarrow X$ be an alteration as in the assumption. One can define a canonical norm map $\Nm:f_*\GG_m\lra \GG_m$ such that the composition 
$$\GG_m\lra f_*\GG_m \lra \GG_m$$
is equal to the multiplication by $[K(X_1):K(X)]$. Then $\Nm$ induces a natural map
$$\Nm:H^2(X_1,\GG_m)=H^2(X, f_*\GG_m)\lra H^2(X,\GG_m).$$
By definition, the composition of $\Nm$ and $f^*$
$$H^2(X,\GG_m)\stackrel{f^*}{\lra}H^2(X_1,\GG_m)\stackrel{\Nm}{\lra}H^2(X,\GG_m)$$
is the multiplication by $[K(X_1):K(X)]$. It follows that $\Nm$ has a cokernel of finite exponent. Consider the following diagram
\begin{displaymath}
\xymatrix{
	\Br(X_1) \ar[r] \ar[d]^{\Nm} & \Br((X_1)_{k^s})^{G_k}\ar[d]^{\Nm} \\
   \Br(X)\ar[r]  & \Br(X_{k^s})^{G_k}}
\end{displaymath}
The diagram commutes (cf., e.g., \cite[Prop. 2.2]{ABBG}).
Since all vertical maps have cokernels of finite exponent, it suffices to show the first row has a cokernel of finite exponent up to $p$-torsion. Thus, we can assume that $X$ is an open subvariety of a smooth projective variety $\bar{X}$. Consider the following commutative diagram
\begin{displaymath}
\xymatrix{
	\Br(\bar{X})(\non p) \ar[r] \ar[d] & \Br(\bar{X}_{k^s})^{G_k}(\non p)\ar[d] \\
   \Br(X)(\non p)\ar[r]  & \Br(X_{k^s})^{G_k}(\non p)}
\end{displaymath}
It is well-known that the first row has a cokernel of finite exponent (cf. \cite{Yua}). By \cite[Prop. 2.3]{Qin1}, the second column has a cokernel of finite exponent. Thus, the second row has a cokernel of finite exponent. This completes the proof.
\end{proof}
\begin{thm}\label{noncomp}
Let $X$ be a smooth geometrically connected variety over a finite field $k$ of characteristic $p$. Assuming that $\Br(X_{k^s})^{G_k}(\ell)$ is finite for some prime $\ell\neq p$, then $\Br(X_{k^s})^{G_k}(\non p)$ is finite.
\end{thm}
\begin{proof}
By \cite[Thm. 3.5]{Qin2}, $\Br(X_{k^s})^{G_k}(\ell)$ is finite for all $\ell \neq p$. By \cite[Prop. 2.3]{Qin1}, shrinking $X$ will not change the question. Thus, we can assume that $X$ satisfies the assumption in Lemma \ref{surj}. So the natural map $\Br(X)(\ell) \rightarrow \Br(X_{k^s})^{G_k}(\ell)$ is surjective for $\ell \gg 0$. It suffices to show that $\Br(X)(\ell)$ is divisible for $\ell \gg 0$. Consider the exact sequence
$$0\lra \Br(X)/\ell^n\lra H^3(X,\ZZ/\ell^n(1)) \lra H^3(X,\GG_m)[\ell^n]\lra 0.$$
Taking limit, we get an injection
$$\Br(X)\hat{\otimes}\ZZ_\ell\hookrightarrow H^3(X,\ZZ_\ell(1)).$$
Since $\Br(X)$ is a torsion abelian group of cofinite type, we have
$$\Br(X)\hat{\otimes}\ZZ_\ell \cong \Br(X)(\ell)/(\Br(X)(\ell))_{div}.$$
Thus, it suffices to show $H^3(X,\ZZ_\ell(1))_{tor}=0$ for $\ell \gg 0$. By the Hochschild-Serre spectral sequence
$$E_2^{p,q}=H^p_{cts}(G_k,H^q(X_{k^s}, \ZZ_\ell(1)))\Rightarrow H^{p+q}(X, \ZZ_\ell(1))$$
and the vanishing of $E_2^{p,q}$ for $p\geq 2$, we get an exact sequence
$$0\lra H^2(X_{k^s}, \ZZ_\ell(1))_{G_k}\lra H^3(X, \ZZ_\ell(1))\lra H^3(X_{k^s}, \ZZ_\ell(1))^{G_k}\lra 0.
$$
By Theorem \ref{bigthm}, the third term vanishes for $ \ell \gg 0$. It suffices to show 
$$H^2(X_{k^s}, \ZZ_\ell(1))_{G_k}[\ell]=0 \ \mathrm{for} \ \ell \gg 0 .
$$
By \cite[Lem. 3.4]{Qin2}, there is a split exact sequence of $G_k$-representations
$$0\lra \Pic(X_{k^s})\otimes_\ZZ\QQ_\ell \lra H^2(X_{k^s}, \QQ_\ell(1))\lra V_\ell\Br(X_{k^s})\lra 0.
$$
Notice that $H^2(X_{k^s}, \ZZ_\ell(1))_{tor}=0$ for $\ell \gg 0$ by Lemma \ref{notor}. Set $M_\ell=H^2(X_{k^s}, \ZZ_\ell(1))$, by a same argument as in the proof of Lemma \ref{lem1}, $M=(M_\ell, \ell\in I)$ is a compatible system of $G_k$-modules. By the exact same argument of the proof of the second claim in  Lemma \ref{lem1}, we have  $(M_\ell)_{tor}=0$ for $\ell \gg 0$. This completes the proof.
\end{proof}
\begin{rem}
It is possible that Theorem \ref{noncomp} still holds for any finitely generated field $k$ of positive characteristic. The tricky part is to show the $\ell$-independence of the finiteness of $\Br(X_{k^s})^{G_k}(\ell)$. Assuming resolution of singularity for varieties over fields of positive characteristic, then Theorem \ref{noncomp} can be derived directly from Theorem \ref{thm1}.
\end{rem}
\subsection{Weil restriction}
To reduce Theorem \ref{thm2} to Theorem \ref{noncomp}, we need a technique of Weil restriction for abelian varieties. Following Milne's paper \cite{Mil2}, we will show that the Tate-Shafarevich group for an abelian variety does not change under Weil restriction. In fact, $L$-function for an abelian variety does not change under Weil restriction either. So the same idea can give a new proof of Theorem 1.9 in \cite{Qin2}.
\begin{prop}\label{wrest}
Let $f:\CZ\rightarrow \CY$ be a finite etale morphism between smooth geometrically connected varieties over a finite field $k$ of characteristic $p$. Write $K$ (resp. $L$) for the function field of $\CY$ (resp. $\CZ$). Let $B$ be an abelian variety over $L$ and $A$ be its Weil restriction to $K$. Assuming that $A$ and $B$ extend to abelian schemes over $\CY$ and $\CZ$ respectively. Then
$$\Sha_{\CY}(A)(\ell)\cong \Sha_{\CY}(B)(\ell)$$
for any $\ell \neq p$.
\end{prop}
\begin{proof}
Let $\CB$ (resp. $\CA$) denote the etale sheaf on $\Spec L$ associated to the Galois module $B(L^s)$ (resp. $B(K^s)$). Let $f$ denote the induced morphism $\Spec L\rightarrow \Spec K$. Then the sheaf $f_*\CB$ is isomorphic to the sheaf $\CA$. This can be seen by checking the equivalence between Frobenius reciprocity and adjoint property of $f^*$ and $f_*$. Let $\SA \rightarrow \CY$ (resp.  $\SB \rightarrow \CZ$) be an abelian scheme extending $A\rightarrow \Spec K$ (resp. $B\lra \Spec L$). Let $i$ (resp. $j$) denote the morphism $\Spec K\rightarrow \CY$ (resp. $\Spec L\rightarrow \CZ$). By \cite[Thm. 3.3]{Kel1}, we have $ \SA\cong i_*\CA$ and $\SB\cong j_*\CB$. Thus,
$f_*\SB\cong \SA.$ By \cite[Lem. 5.2]{Qin2}, we have 
$$\Sha_{\CY}(A)(\ell)\cong H^1(\CY, \SA)(\ell) \ \mathrm{and} \ \Sha_{\CZ}(B)(\ell)\cong H^1(\CZ, \SB)(\ell).$$
Since $f$ is finite, we have $H^1(\CZ, \SB)\cong H^1(\CY, f_*\SB)$. Then, the claim follows from 
$$
H^1(\CZ, \SB)\cong H^1(\CY, f_*\SB)\cong H^1(\CY, \SA).
$$
\end{proof}
\begin{rem}
By the same argument, we have $f_*V_\ell \SB\cong V_\ell \SA$. Thus, $f_*V_\ell \SB(-1)\cong V_\ell \SA(-1)$. It follows $L(\CY, V_\ell\SA(-1), s)=L(\CY, f_*V_\ell\SB(-1), s)=L(\CZ, V_\ell\SB(-1), s)$. Since $A(K)=B(L)$ by definition, so the BSD conjecture for $A$ is equivalent to the BSD conjecture for $B$.
\end{rem}
\subsection{Proof of Theorem \ref{thm2}}
\begin{lem}\label{flem}
Let $K=k(t_1,...,t_m)$ be a purely transcendental extension of a finite field $k$. Let $A$ be an abelian variety over $K$. Then Theorem \ref{thm2} holds for $A$.
\end{lem}
\begin{proof}
Let $\CY$ be a smooth geometrically connected variety over $k$ with a function field $K$. Let $A^t$ denote $\Pic^0_{A/K}$. Then $A$ can be identified with $\Pic^0_{A^t/K}$. By \cite[Lem. 5.2]{Qin2}, without loss of generality, we can shrink $\CY$ such that $A^t$ extends to an abelian scheme $\pi:\SA^t\rightarrow\CY$. By \cite[Thm. 1.11]{Qin2}, there is an exact sequence
$$0\lra V_\ell\Br(\CY_{k^s})^{G_k}\lra \Ker (V_\ell\Br(\SA^t_{k^s})^{G_k}\lra V_\ell\Br(A^t_{K^s})^{G_K})\lra V_\ell\Sha_{\CY}(A)\lra 0.$$
Since $T^1(X,\ell)$ holds for rational varieties and abelian varieties(Zarhin's theorem), we have $V_\ell(\CY_{k^s})^{G_k}=0$ and $V_\ell\Br(A^t_{K^s})^{G_K}=0$. Assuming that $\Sha_{\CY}(A)(\ell)$ is finite, then $V_\ell\Sha_{\CY}(A)=0$. Thus, $V_\ell\Br(\SA^t_{k^s})^{G_k}=0$. So $\Br(\SA^t_{k^s})^{G_k}(\ell)$ is finite. By Theorem \ref{noncomp}, $\Br(\SA^t_{k^s})^{G_k}(\ell)=0$ for $\ell \gg 0$. In fact, the proof of  \cite[Thm. 1.11]{Qin2} can imply that there is an exact sequence for $\ell \gg 0$
$$0\lra \Br(\CY_{k^s})^{G_k}(\ell)\lra \Ker (\Br(\SA^t_{k^s})^{G_k}(\ell)\lra \Br(A^t_{K^s})^{G_K}(\ell))\lra \Sha_{\CY_{k^s}}(A)^{G_k}(\ell)\lra 0.$$
It follows that $\Sha_{\CY_{k^s}}(A)^{G_k}(\ell)=0$ for $\ell \gg 0$. Next, we will show that the natural map
$$\Sha_{\CY}(A)\lra \Sha_{\CY_{k^s}}(A)^{G_k}$$
has a finite kernel. Let $K^\prime$ denote $Kk^s$. There is an exact sequence
$$0\lra H^1(G_k, A(K^\prime))\lra H^1(K,A)\lra H^1(K^\prime, A).$$
It suffices to show that $H^1(G_k, A(K^\prime))$ is finite. By Lang-N\'eron Theorem, the quotient
$$A(K^\prime)/\Tr_{K/k}(A)(k^s)$$
is a finitely generated abelian group. Thus, $H^1(G_k,A(K^\prime)/\Tr_{K/k}(A)(k^s))$ is finite. By Lang's theorem, $H^1(G_k,\Tr_{K/k}(A)(k^s))=0$. Thus, 
$H^1(G_k, A(K^\prime))$ is finite. So
$$\Sha_{\CY}(A)(\ell)\lra \Sha_{\CY_{k^s}}(A)^{G_k}(\ell)$$
is injective for $\ell \gg 0$. It follows that $\Sha_{\CY}(A)(\ell)=0$ for $\ell \gg 0$. This completes the proof.
\end{proof}
Proof of Theorem \ref{thm2}.
\begin{proof}
Since $k$ is perfect, so $K/k$ is separably generated, i.e. there exist algebraic independent elements $t_1,..., t_m$ in $K$ such that $K/k(t_1,...,t_m)$ is a finite separable extension. Let $M$ denote $k(t_1,...,t_m)$. Let $B$ be the Weil restriction of $A$ to $M$. By Proposition \ref{wrest}, it suffices to prove the claim for $B$. By Lemma \ref{flem}, the claim holds for $B$. This completes the proof.
\end{proof}

%------------------------------------------------------
\end{document}